\documentclass[12pt]{amsart}
\usepackage[margin=1in]{geometry}

\usepackage{amsmath}
\usepackage{amsthm}
\usepackage{mathtools}
\usepackage{xcolor}
\usepackage{enumerate}
\usepackage{enumitem} 
\usepackage{tikz}
\usetikzlibrary{arrows,shapes,trees,decorations.markings} 

\newcommand{\R}{{\mathbb R}}
\newcommand{\Z}{{\mathbb Z}}
\newcommand{\bZp}{{{\mathbb Z}^+}}

\DeclarePairedDelimiter\abs{\lvert}{\rvert} 
\DeclarePairedDelimiter\norm{\lVert}{\rVert} 

\newcommand{\supp}[1]{\text{supp} \left( #1 \right)}

\newcommand{\mc}{\text{, }}
\newcommand{\md}{\text{.}}

\usepackage{blindtext}
\usepackage{comment}

 \newtheorem{theorem}{Theorem}[section]
 \newtheorem{Def}[theorem]{Definition}
 \newtheorem{proposition}[theorem]{Proposition}
 \newtheorem{lemma}[theorem]{Lemma}

 \newtheorem{Example}[theorem]{Example}

 \numberwithin{equation}{section}
 \renewcommand{\rm}{\normalshape}
 
\begin{document}

\title  {Riesz bases of exponentials for multi-tiling measures}


\author{Chun-Kit Lai}

\address{Department of Mathematics, San Francisco State University,
1600 Holloway Avenue, San Francisco, CA 94132.}
 \email{cklai@sfsu.edu}

\author{Alexander Sheynis}
\address{Department of Mathematics, San Francisco State University,
1600 Holloway Avenue, San Francisco, CA 94132.}
 \email{asheynis@mail.sfsu.edu}

\subjclass[2010]{Primary 42B15}
\keywords{Riesz bases, Fourier frames, Multi-tiles, Spectral Measures }

\begin{abstract}
Let $G$ be a closed subgroup of $\R^d$ and let $\nu$ be a Borel probability measure admitting a Riesz basis of exponentials with frequency sets in the dual group $G^{\perp}$. We form a multi-tiling measure $\mu = \mu_1+...+\mu_N$ where $\mu_i$ is translationally equivalent to $\nu$ and different $\mu_i$ and $\mu_j$ have essentially disjoint support. We obtain some necessary and sufficient conditions for $\mu$ to admit a Riesz basis of exponentials . As an application,  the square boundary, after a rotation, is a union of two fundamental domains of $G = \Z\times \R$ and can be regarded as a multi-tiling measure. We show that, unfortunately, the square boundary does not admit a Riesz basis of exponentials of the form as a union of translate of discrete subgroups $\Z\times \{0\}$. This rules out a natural candidate of potential Riesz basis for the square boundary. 
\end{abstract}

\maketitle

\section{Introduction}
This paper is about multi-tiling Riesz basis spectrality in singular measure and non-discrete subgroup settings. Throughout the paper, $\Lambda$ is a countable set in $\R^d$ and we denote by $E(\Lambda) = \{e^{2\pi i \lambda\cdot x}: \lambda\in \Lambda\}$ the exponential system with frequency set $\Lambda$.  Let $\nu$ be a  Borel probability measure on $\R^d$. We say that $\nu$ is a {\bf spectral measure} if there exists an exponential orthonormal basis $E(\Lambda)$ for $L^2(\nu)$ and $\Lambda$ is called a {\bf spectrum} for $\nu$. $\nu$ is called a {\bf frame-spectral measure} if $E(\Lambda)$ forms a frame for $L^2(\nu)$. i.e. there exists $0<A\le B<\infty$ such that 
$$
A \int |f|^2 d\nu \le \sum_{\lambda\in \Lambda} \left|\int f(x) e^{2\pi i \lambda\cdot x}d\nu (x)\right|^2 \le B \int |f|^2 d\nu.
$$
$\nu$ is called a {\bf Riesz-spectral measure} if $E(\Lambda)$ forms a {\bf Riesz basis} for $L^2(\nu).$ It means that $E(\Lambda)$ is both a frame and a Riesz sequence. i.e.  there exists $0<A\le B<\infty$ such that 
$$
A \sum_{\lambda\in \Lambda} |c_{\lambda}|^2 \le \int \left|\sum_{\lambda\in \Lambda}  c_{\lambda}e^{2\pi i \lambda\cdot x}\right|^2 d\nu (x) \le B \sum_{\lambda\in \Lambda} |c_{\lambda}|^2.
$$
For general theory of frame and Riesz bases on Hilbert space, one can refer to \cite{Chr,H}.

\subsection{Connection to Tiling.} Spectral measures have a close, but yet delicate connection to tiling. Fuglede \cite{Fug} initiated the study of spectral sets (i.e. the normalized Lebesgue measure on the measurable set is a spectral measure). He proposed the conjecture that  a set is spectral if and only if it is a translational tile. The conjecture in its full generality was finally disproved by Tao \cite{T}, followed by Kolountzakis and Matolcsi \cite{KM06}. It remains open in dimension 1 and 2. 

\medskip

Despite being incorrect in its full generality, Fuglede's conjecture remains true under some natural additional assumptions. Fuglede \cite{Fug} proved that the conjecture is true if we assume that the spectrum is a lattice on $\R^d$, in which the spectral set will be a translational tile with tiling set being the dual lattice (see \cite{I} for a simpler proof). Recently, Lev and Matolcsi proved that Fuglede's conjecture is true if we assume the set is convex \cite{LM2022}.

\medskip

  Apart from spectral sets, sets admitting Riesz bases of exponentials also has a close connection to multi-tiling. A set $\Omega$ is said to be a {\bf multi-tile } by a lattice $\Gamma$ in $\R^d$ if there exists an integer $k>0$ such that 
  $$
  \sum_{\gamma\in\Gamma} {\bf 1}_{\Omega}(x+\gamma) = k. \ \mbox{a.e.}
$$
When $k=1$, $\Omega$ is a translational tile. For a group $G$, We will denote by $G^{\perp}$ its dual group.  The following theorem was first proved by Kolountzakis \cite{K15} (see \cite{GLev} for an earlier result using quasicrystals).  

\begin{theorem}
    Let $\Omega$ be a bounded multi-tile by a discrete full-rank lattice $\Gamma$ on $\R^d$. Then there exists $t_1,...,t_k$ such that $E\left(\bigcup_{i=1}^k (\Gamma^{\perp}+t_i)\right)$ is a Riesz basis of exponentials for $L^2(\Omega)$. 
\end{theorem}

The boundedness of $\Omega$ cannot be removed in the statement \cite{AAC}.  An intensive study for the unbounded multi-tiles was carried out in the subsequent papers \cite{CC, CKM} and a satisfactory answer was given using the Bohr topology \cite{CKM}. Following the terminology used in \cite{CC}, we will say that a set $\Lambda$ is {\bf structured} if we can write 
$$
\Lambda = \bigcup_{i=1}^k (\Gamma+t_i)
$$
where $\Gamma$ is a discrete subgroup of $\R^d$ and $t_i$ is some distinct points of $\R^d$ that are distinct residue modulo $\Gamma$.

\medskip

    The first singular spectral measures, not being an atomic measure, was first discovered by Jorgensen and Pedersen \cite{JP98}. They found that the middle-fourth Cantor measure is a spectral measure with a sparse spectrum inside $\Z$. However, the usual commonly-known middle-third Cantor measure is not a spectral measure. There have been serious studies of fractal spectral measures since the first discovery of such measures (see \cite{DLW} for a survey).  The construction of singular spectral measures also possess delicate connections to tiling. One can refer to \cite{GL} for more details. Beside spectral measures of fractal types, the surface measures on manifold and polytopes were also studied by Lev and Iosevich-Liu-Lai-Wyman \cite{Lev,ILLM}. It was shown that all surface measures of polytopes admits a frame of exponentials, while the surface measure of a sphere does not admit any such frames.

\medskip

\subsection{Main Contribution.} The purpose of this paper is to prove a type of  singular multi-tiling measure is Riesz spectral and develop the related theory as a generalization to the classical multi-tile setting. The motivation of this study was first initiated  by an open-ended problem raised by Iosevich to the first-named author. Iosevich asked if there could be a relationship between the spectrality on the boundary singular measures and the spectrality of the domain. If such connection exists and we know that the boundary of the disk has no frame of exponentials, there may be an answer to the widely open problem that if the unit disk admits a Riesz basis of exponentials.   Initiated by this question and the question asking the circle seems difficult, we investigate the following simpler but natural question.

\medskip

{\bf (Qu):} Does the boundary of the unit square admit a Riesz basis of exponentials?

\medskip

 Unfortunately, we are not able to obtain an affirmative answer to the proposed question. However, we observe that boundary of a unit square, can be written as a union of two fundamental domains of the closed subgroup $\Z\times \R$ (See Figure \ref{fig:SqBdryNoRB}). This leads us to generalizing the  multi-tiling into a singular measure setting.

	\begin{figure}[h]
		\begin{center}
			\resizebox{8cm}{!}
			{
				\begin{tikzpicture}[decoration={markings,
					}]
					\draw[line width=1.5pt,-] (-5,0) -- (5,0) coordinate (xaxis);
					\draw[line width=1.5pt,-] (0,-3) -- (0,3) coordinate (yaxis);
					
					\foreach \x in {-3,0,3}
					{
						\draw[dashed] (\x , -3) -- (\x, 3);
					}
					
					\coordinate (0) at (0, 0);
					
					\path[draw,line width=1pt,postaction=decorate] (0, 0) -- (1.5, 1.5) -- (3, 0) -- (1.5, -1.5) -- (0, 0);
					
					\node[below] at (xaxis) {$x$};
					\node[left] at (yaxis) {$y$};
					\node[below left] {$O$};
					\node[below right] at (3, 0) {$1$};
					\node[below left] at (-3, 0){$-1$};
					\node[above] at (1.5, 1.5){$\mu_{1}$};
					\node[below] at (1.5, -1.5) {$\mu_{2}$};
				\end{tikzpicture}
			}
			\caption{The boundary of the square overlaid with $G = \Z \times \R$. $\mu_1$ is the one-dimensional Lebesgue measure on the $L$-shape above the $x$-axis and $\mu_2$ is the corresponding measure below the $x$-axis. the surface measure on the square is $\mu_1+\mu_2$ and their support in a union of two fundamental domains of $\R\times\Z$.}
			\label{fig:SqBdryNoRB}
		\end{center}
	\end{figure}
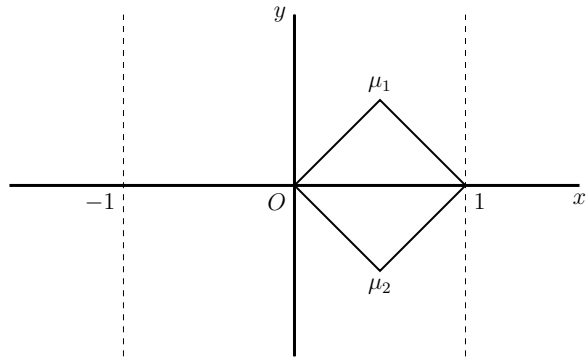

In a similar favor, the authors \cite{LLP} also studied a similar question on the additive space of two line measures. It is still unknown if the Lebesgue measure supported on the ``plus space" given by 
$$
\frac12 \left( m_{[-1/2,1/2]}\times \delta_0 + \delta_0\times m_{[-1/2,1/2]}\right)
$$
is Riesz-spectral. Indeed, after a rotation of $45^{\circ}$, the plus space is also a union of two fundamental domains of a closed subgroup $\Z\times \R$. 

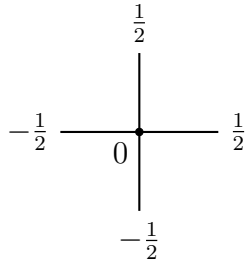
\begin{figure}[h]
  \centering
\begin{tikzpicture}[scale=.7]
  \draw[thick] (-1.5,0) node[anchor=east] {$- \frac 1 2 $}  -- (1.5,0) 
  node[anchor=west] {$\frac 1 2$};
  \draw[thick] (0,-1.5) node[anchor=north] {$- \frac 1 2$} -- (0,1.5) 
  node[anchor=south] {$\frac 1 2$};
  \filldraw[black] (0,0) circle (2pt) node[anchor=north east] {0};
\end{tikzpicture}
\caption{The Plus Space}
  \label{fig_Plus}
\end{figure}

\medskip

We formulate this observation in a generalized multi-tiling perspective (Section 2). For a given Riesz-spectral measure $\nu$  supported in a fundamental domain $X$ and with a Riesz spectrum $\Lambda\subset G^{\perp}$, we will introduce an associated multi-tiling measure and study its Riesz-spectrality with a structured Riesz basis. A complete characterization will be given in Theorem \ref{thm:MainResult} and its proof will be given in Section 3.   This formulation is different from previous work in two senses. First, the group may not be discrete and it is only a closed subgroup in $\R^d$. Second, the measure  $\nu$ does not need to be a spectral set with a lattice spectrum. Instead,  we can start with  a singular measure admitting a Riesz basis of exponentials. This will offer a lot of flexibility that can generate a Riesz-spectral measure. 

\medskip

After studying the new formulation, we obtain several criteria for which we can guarantee the multi-tiling measure admits a structured Riesz basis of exponentials (Section 4). Finally, we show that

\begin{theorem} \label{thm:BndrySqResult}
	The boundary of the square does not admit a Riesz basis of the type $\bigcup_{k = 1}^{N} (\Lambda + t_{k})$ where $\Lambda = \Z \times \{ 0 \}$.
\end{theorem}

The same argument also works for any boundary of any convex polygons as well as the plus space. This result ruled out a natural choice of candidates for the Riesz spectra for these surface measures. It will be an interesting problem to see how to construct a Riesz basis or disprove the existence of the Riesz basis for the square boundary. Some other examples will be given in Section 5. 

\medskip

\section{The multi-tiling measure Setup}
\subsection{Closed subgroup and fundamental domains.} Throughout this paper,   $G$ is a closed additive subgroup of $\R^{d}$ and $G$ may not be discrete. A {\bf fundamental domain} of $G$ is a Borel subset of ${\mathbb R}^d$ consisting of exactly one coset representative in ${\mathbb R}^d/G$. It is well-known that if $G$ is a discrete subgroup, we can easily choose a fundamental domain for $G$. The following lemma first classifies all closed subgroups of $\R^d$. Since we do not find a documented proof, we provide here for completeness.

\begin{lemma} \label{lem:RdSubgroup}
	Let $G$ be a closed subgroup of $\R^{d}$. Then there exists a vector subspace $H$ and a discrete subgroup $\Gamma \subset H^{\perp}$ such that
	
	\begin{equation}\label{H+Gamma}
		G = H \oplus \Gamma.
	\end{equation}
\end{lemma}
\begin{proof}
	Let $\epsilon > 0$. Consider the vector subspace
	$$
		H_{\epsilon}= \mbox{Span}_{\R} \left\{ {\bf v}: {\bf v} \in B(0, \epsilon) \cap G \right\}.
	$$
Note that $H_{\epsilon}$ form a decreasing sequence of subspaces. There exists $\epsilon_{\ast}$ such that if $\epsilon \leq \epsilon_{\ast}$, then $H_{\epsilon} = H_{\epsilon_{\ast}}$ and we now let $H = H_{\epsilon_{\ast}}$ with dim $H = s$. Define $\Gamma = H^{\perp} \cap G$. We will prove this lemma by establishing the following three claims.
	
	\begin{enumerate}
		\item The subspace $H \subset G$.
		\item $\Gamma$ is a discrete subgroup.
		\item $G = H \oplus \Gamma$. \\
	\end{enumerate}
	
	To see claim (1), Given $n \in \bZp$ such that $1/n < \epsilon_{\ast}$, we let $h \in H = H_{1/n}$. Then we can write $h = \sum_{j=1}^{s} a_{j} {\bf v}_{j}$ where $\{{\bf v}_{1},...,{\bf v}_{s}\}$ is a basis for $H$ with $\|{\bf v}_{i}\|\le 1/n$. Note that as $G$ is an additive subgroup, 
$$	
		\mbox{Span}_{\Z} \left\{ {\bf v}: {\bf v} \in B(0 , 1/n) \cap G \right\} \subset G \cap H
$$	
 where $\mbox{Span}_{\Z}$ means the coefficients of the linear combinations are integers. Define $\widetilde{h}_{n} = \sum_{j = 1}^{s} \lfloor a_{j} \rfloor {\bf v}_{j}$ where $\lfloor x \rfloor$ is the usual floor function that outputs the largest integer smaller than $x$. Then $\widetilde{h_n} \in G$ and
	
	\begin{equation*}
		|h - \widetilde{h}_{n}| \leq \sum_{j=1}^s |a_j-\lfloor a_j\rfloor|\|{\bf v}_j\|\le \frac{s}{n}.
	\end{equation*}
	
	\noindent Taking $n \to \infty$ and noting that $G$ is closed, we have $h \in G$. This shows claim (1). 
	
	To prove claim (2), we note that $\Gamma$ is a subgroup of $\R^{d}$ and we just need to show that it is discrete. Suppose to the contrary that it is not discrete. Then there exists a sequence of $\gamma_{n} \in \Gamma$ such that $\gamma_{n} \to \gamma$, but $\gamma_{n} \neq \gamma$ for all $n \in \bZp$. As $G$ and $H^{\perp}$ are closed, $\gamma \in \Gamma$. But this implies that $\gamma_{n} - \gamma \in H$ for $n$ sufficiently large. As $\gamma_{n} - \gamma \in H^{\perp}$ and $H \cap H^{\perp} = \{ 0 \}$, this forces $\gamma_{n} = \gamma$ for $n$ sufficiently large which is a contradiction. Hence, $\Gamma$ must be discrete. \\
	
	Claim (3) follows directly by noting that any $g \in G$ can be uniquely written as $g = h + h^{\perp}$ where $h \in H$ and $h^{\perp} \in H^{\perp}$. But then $h^{\perp} = g - h \in G$ as $G$ is a group and $H \subset G$, meaning that $h^{\perp} \in \Gamma$. This completes the proof. \qed \par
\end{proof}

\medskip
Because of the above lemma, we have the following proposition justifying the existence of fundamental domain in any closed subgroups of $\R^d$. 

\begin{proposition}
    Let $G$ be a closed subgroup of $\R^d$ with the decomposition in (\ref{H+Gamma}). Then a fundamental domain of $\Gamma$ in  $H^{\perp}$ is a fundamental domain of $G$. 
\end{proposition}

\begin{proof}
    The proof follows directly from the fact that $\R^d = H\oplus H^{\perp}$ and $H^{\perp} = \left(H^{\perp}/\Gamma\right)\oplus \Gamma$. \qed \par
\end{proof}

\subsection{Dual Group.} The dual group of a closed subgroup $G$ is defined to be 
$$
G^{\perp} = \left\{ x\in \R^d: x\cdot g \in \Z \  \forall g\in G\right\}.
$$
If $G$ is a discrete full-rank subgroup on $\R^d$, then we can write $G = A(\Z^d)$ for some invertible matrix $A$ and it is well-known that $G^{\perp} = (A^{\top})^{-1}(\Z^d)$. In general, if we write $G = H\oplus \Gamma$ where $H$ is a closed non-trivial subspace of $\R^d$ and $\Gamma$ is a discrete lattice on $H^{\perp}$. Then, $G^{\perp}$ is the dual lattice of $\Gamma$ in $H^{\perp}$. Therefore, it is a discrete subset in $\R^d$.  For example, if $G = \R\times \Z$ in $\R^2$. Then 
$$
G = (\R\times\{0\}) \oplus (\{0\}\times \Z), \ \mbox{and} \ G^{\perp} = \{0\}\times \Z.
$$

\subsection{Multi-tiling measures} We fix a fundamental domain $D$ of $G$. Suppose that $\nu$ is a finite Borel measure whose essential support is in $D$. By essential support, we mean a Borel set $X_{\nu}$ such that 
$$
\nu (X_{\nu}) = \nu (\R^d).
$$
We will denote it by $\mbox{supp}(\nu)$. 
Let $D'$ be another fundamental domain of $G$. Then there exists a unique map $\pi: D' \to D$ such that $\pi (x) = x+ g(x)$ where $g(x)$ is the unique element in $G$ such that $\pi(x)\in D$. As $D$ and $D'$ are both fundamental domains, $\pi$ is a bijective map. The following lemma can be done through  a direct calculation, so we will omit the proof.

\begin{lemma}
    Let $\nu$ be a finite Borel measure supported on a fundamental domain $D$ of a closed subgroup $G$. Let $D'$ be another fundamental domain  of $G$ and 
  $$
  \mu(E) = \nu (\pi (E)).
$$
Suppose that $\nu$ is a frame-spectral/ Riesz-spectral  measure with a frame-spectrum/Riesz spectrum $\Lambda\subset G^{\perp}$. Then $\mu$ is also a frame-spectral/ Riesz-spectral  measure with a frame-spectrum/Riesz spectrum $\Lambda\subset G^{\perp}$.
\end{lemma}

\medskip



\begin{Def}\label{Def-multi-tiling} {\rm Let $G$ be a closed subgroup. Fix a fundamental domain $D$ and a finite Borel measure $\nu$ essentially supported on $D$. Let $\mu_1,...,\mu_N$ be finite Borel measures such that for each $k = 1,...,N$,  the essential support, denoted by $X_k$, of $\mu_{k}$ is a subset of  a fundamental domain $D_k$ of $G$. Define the injective map  $\pi_{k}: X_k \to D$. Assume that for $i=1,...,N$}
\begin{equation} \label{eqn:EqvlntPbMeasure}
	\mu_1 = \nu \circ \pi_1 , \ .... \ , \mu_N = \nu \circ \pi_N
\end{equation}
{\rm Then the {\bf multi-tiling measure} of $\mu_{1},  \ldots , \mu_{N}$ with respect to $G$ is defined to be}
\begin{equation} \label{eqn:MultiTileMeasure}
	\mu = \mu_{1} + \ldots + \mu_{N} 
\end{equation}
{\rm  where $\mu(X_j\cap X_k) = 0$ for all $k \neq j$ (a measure-disjoint condition). }
\end{Def}

We refer $\mu$ as a multi-tiling measure because if $\nu$ is the Lebesgue measure supported on the fundamental domain, $\mu$ is the Lebesgue measure supported on a multi-tile (see Example \ref{eg:R2UnitSqTiling}). In general, the measure $\nu$ and $\mu_i$ do not really tile the space. However, each $\mu_i$ are supported inside a fundamental domain and in most cases, it has properties that behave like translational tiles. For example, a spectral set tiles   ``weakly" (see \cite{LM2022} for details). Certain singular spectral measures can ``tile" the fundamental domain in the sense that one can find another measure $\nu$  that satisfies $\mu\ast\nu = \chi_{[0,1]^d}dx$, the Lebesgue measure on $[0,1]^d$ \cite{GL}. A precise generalized Fuglede's conjecture for singular measures was also formulated in the same paper \cite{GL}.  In view of this, $\mu$ will be the ``multi-tiling version" of $\nu$. Therefore, we choose the name multi-tiling measures to describe $\mu$.  

\begin{Example} \label{eg:R2UnitSqTiling}
	{\rm Let $\nu$ be the Lebesgue measure on the unit $d$-dimensional cube, $[0 \mc 1]^{d}$, which is a fundamental domain of $\Z^{d}$, and let $X_{1} \mc \ldots \mc X_{k}$ be disassembled copies of the fundamental domain so that their intersections have Lebesgue measure zero. Define $\mu_{i}$ to be the Lebesgue measure on $X_{i}$. Then the multi-tiling measure $\mu$ defined in \eqref{eqn:MultiTileMeasure} is the Lebesgue measure supported on the multi-tile $X_{1} \cup \cdots \cup X_{k}$. This is exactly the mult-tile of level $k$ considered by Kolountzakis in \cite{K15}.}
\end{Example}

\begin{Example} \label{eg:CantorMid4th}
	{\rm Consider the middle-fourth Cantor measure $\nu$ studied by Jorgensen and Pedersen in \cite{JP98}. They showed that $\nu$ has a Fourier orthonormal basis whose spectrum is $\Lambda = \{ \sum_{k = 0}^{N - 1} 4^{j} d_{k} : d_{k} \in \{ 0, 1 \} \text{ and } N \in \bZp \}$. Let $G = \Z$ and $D = [0, 1)$. We let $D_1,...,D_k$ be bounded fundamental domains of $G$ and they are mutually disjoint and denote $\pi_i$ by its corresponding projection maps.  We can  form a multi-tiling measure $\mu = \mu_1+...+\mu_k$ where $\mu_i$ are built out of $\nu$ such that $\nu = \mu_{i} \circ \pi_{i}$ for each $i=1,...,k$.} 
\end{Example}

For a given $x \in \mbox{supp}{(\nu)} \subset D$, there exists a unique $g_{j}(x)$ such that $x + g_{j}(x) \in \mbox{supp}{(\mu_{j})}$. We collect all these $g_{j}(x)$ in
\begin{equation} \label{eqn:PbMapElementCollection}
	{\mathcal G}_{x} = \{ g_{1}(x) \ldots  g_{N}(x) \}. 
\end{equation}
Define ${\mathfrak D}_{x}: \R^{N d} \to {{\mathbb C}}$ by
\begin{equation} \label{eqn:DeterminantFunction}
	{\mathfrak D}_{x} ({\bf t}) = \det(M_{x}({\bf t})) 
 \end{equation}
 where ${\bf t} = (t_{1} , \ldots , t_{N})$ and  $M_{x}({\bf t}) $ is an $N\times N$ matrix whose entries are
 $$
 (M_{x}({\bf t}))_{j, k} = e^{2\pi i  t_{k}\cdot g_{j}(x)}. 
 $$
 We can now state the main theorem, which generalize some of  the results appeared in \cite{K15,AAC,CC,CKM}. 

\begin{theorem} \label{thm:MainResult}
	Suppose that  $\mu$ is a multi-tiling measure with respect to a closed subgroup $G$ given in Definition \ref{Def-multi-tiling}. Then the following are equivalent:
	
	\begin{enumerate}
		\item There exist ${\bf t} \in \R^{N d}$ such that ${\bf t} = (t_{1},  \ldots,  t_{N})$ with each $t_{k} \in \R^{d}$ for which $|{\mathfrak D}_{x}({\bf t})| \geq \epsilon_{0} > 0$ for $\nu$- a.e. $x$.
		
		\item If $\nu$ is frame spectral with a spectrum $\Lambda$ in $G^{\perp}$, then $\mu$ is frame spectral with frame  spectrum $\bigcup_{k = 1}^{N} (\Lambda + t_{k})$.
		
		\item If $\nu$ is a Riesz basis with spectrum $\Lambda$ in $G^{\perp}$, then $\mu$ is a Riesz spectral with a Riesz spectrum $\bigcup_{k = 1}^{N} (\Lambda + t_{k})$.
	\end{enumerate}
\end{theorem}

The proof of this theorem will be based on verifying the frame and Riesz sequence inequalities under the multi-tiling measure assumption. It has a slightly different perspective from the previous proofs in \cite{K15,AAC}. It will be given in the next section.

\section{Proof of Theorem \ref{thm:MainResult}}

Let us first recall the Hadamard's inequality for determinant of square matrices which states that if $A$ is a square matrix whose column vectors are ${\bf v}_1,...,{\bf v}_N$, then 
$$
|\det (A)|\le \prod_{i=1}^N \|{\bf v}_i\|.
$$
In particular, if $A$ has only unimodular entries, then $|\det (A)|\le N^{N/2}$.

\begin{lemma}
    There exists $\epsilon_0>0$ such that $|{\mathfrak D}_{x}({\bf t})| \geq \epsilon_{0} > 0$ for $\nu$ a.e. $x$ if and only if there exists $c_{0} > 0$ such that  
	\begin{equation} \label{eqn:DetEquivCondition}
		\|M_{x} ({\bf t}) {\bf v}\|^{2} \geq c_{0}^{2} \|{{\bf v}}\|^{2} \ \mbox{for} \ \nu \ \mbox{a.e.} \ x.
	\end{equation}
\end{lemma}

\begin{proof}
Suppose that  we have (\ref{eqn:DetEquivCondition}). For those $x$ such that (\ref{eqn:DetEquivCondition}) holds,   if $\lambda$ is an eigenvalue of $M_x({\bf t})$, we can plug the eigenvector into (\ref{eqn:DetEquivCondition}). We have immediately that $|\lambda|\ge c_0$. As determinant is a product of eigenvalues, $|{\mathfrak D}_{x}({\bf t})| \geq \epsilon_{0}: = c_0^N> 0$.

\medskip

Conversely, we suppose the determinant condition holds. From the Hadamard inequality and $M_x({\bf t})$ has only unimodular entries, we know that there exists $B>0$ independent of $x$ such that $\|M_{x} ({\bf t}) {\bf v}\|^{2} \leq B \|{{\bf v}}\|^{2}$. Let $0<\sigma_1(x)\le\sigma_2(x) \le ...\le \sigma_N(x)$ be the eigenvalues of $M_{x} ({\bf t})^{\ast}M_{x} ({\bf t})$. We have 
$$
\epsilon_0\le |{\mathfrak D}_x({\bf t})|  = \det (M_{x} ({\bf t})^{\ast}M_{x} ({\bf t}))= \sigma_1(x)...\sigma_N(x)\le B^{N-1}\sigma_1(x).
$$
This shows that $\sigma_1(x)\ge B^{-(N-1)}\epsilon_0$. This shows that (\ref{eqn:DetEquivCondition}) holds, completing the proof. \quad$\Box$
\end{proof}

\medskip

\noindent{\it Proof of Theorem \ref{thm:MainResult}.} For simplicity of notation, we write $e_{\lambda}(x)  = e^{2\pi i \lambda\cdot x}$ in this proof. 

\medskip

	\noindent (1 $\implies$ 2) \indent Consider $\cup_{k = 1}^{N} ( \Lambda + t_{k})$, we have
		{\small
			\begin{align*}
				\sum_{k = 1}^{N} \sum_{\lambda \in \Lambda} \abs*{\int f(x) e_{- (\lambda + t_{k})}(x) \ d\mu(x)}^{2} & = \sum_{k = 1}^{N} \sum_{\lambda \in \Lambda} \abs*{\sum_{j = 1}^{N} \int f(x) e_{- (\lambda + t_{k})}(x) \ d\mu_{j}(x)}^{2} \\
				& = \sum_{k = 1}^{N} \sum_{\lambda \in \Lambda} \abs*{\sum_{j = 1}^{N} \int f(\pi_{j}^{-1}(x)) e_{- (\lambda + t_{k})}(\pi_{j}^{-1}(x)) \ d\nu(x)}^{2}
			\end{align*}
		}
	Moreover, $\pi_{j}^{-1}(x) = x + g_{j}(x)$, where $x \in \operatorname{supp}(\nu)$ and $g_{j}(x) \in G$, so $(\lambda + t_{k}) \cdot \pi_{j}^{-1}(x) = (\lambda + t_{k}) \cdot (x + g_{j}(x)) = (\lambda + t_{k}) \cdot x + t_{k} \cdot g_{j}(x) + \lambda \cdot g_{j}(x)$, where the last term is in $\Z$ since $\lambda\in G^{\perp}$. Continuing from above, we then get
		
		{\footnotesize
			\begin{align*}
				\sum_{k = 1}^{N} \sum_{\lambda \in \Lambda} \abs*{\int f(x) e_{- (\lambda + t_{k})}(x) \ d\mu(x)}^{2} & = \sum_{k = 1}^{N} \sum_{\lambda \in \Lambda} \abs*{\sum_{j = 1}^{N} \int f(\pi_{j}^{-1}((x)) e_{- (\lambda + t_{k})}(x) e_{- t_{k}}(g_{j}(x)) \ d\nu(x)}^{2} \nonumber \\
				& = \sum_{k = 1}^{N} \sum_{\lambda \in \Lambda} \abs*{\int \left[ \sum_{j = 1}^{N} f(\pi_{j}^{-1}(x)) e_{- t_{k}}(g_{j}(x)) e_{- t_{k}(x)} \right] e_{- \lambda}(x) \ d\nu(x)}^{2} \nonumber \\
				& \asymp \sum_{k = 1}^{N} \int \abs*{ \sum_{j = 1}^{N} f(\pi_{j}^{-1}(x)) e_{- t_{k}}(g_{j}(x)) e_{- t_{k}}(x)}^{2} \ d\nu(x) \nonumber
			\end{align*}
		}
		 where the frame inequality was used in the last line. Now noting that $e_{- t_{k}}(x)$ has no dependence on the index $j$ and has unit modulus and that ${\bf F}(x) = (f(\pi_{1}^{-1}(x) \mc \ldots \mc f(\pi_{N}^{-1}(x))^{\top}$,
		
		\begin{align}			
			\sum_{k = 1}^{N} \sum_{\lambda \in \Lambda} \abs*{\int f(x) e_{- (\lambda + t_{k})}(x) \ d\mu(x)}^{2} & = \sum_{k = 1}^{N} \int \abs*{ \sum_{j = 1}^{N} f(\pi_{j}^{-1}(x)) e_{- t_{k}}(g_{j}(x))}^{2} \ d\nu(x) \nonumber \\
			& = \int \norm*{ M_{x}({\bf t}) {\bf F}(x)}^{2} \ d\nu(x) \nonumber \\
			& \asymp \int \norm*{{\bf F(x)}^{2}} \ d\nu(x) \label{eqn:DesiredFrameCondition} \mc
		\end{align}
		 By the hypothesis, the determinant condition given by (a), we have the desired frame inequality which follows from \eqref{eqn:DetEquivCondition} ensuring the uniform bound.
		
\noindent (1 $\implies$ 3) Recall that the pullback map is injective. Again, we will go through a computation showing the bound condition of a Riesz sequence as an asymptotic containment. Now consider $\cup_{k = 1}^{N} ( \Lambda + t_{k})$, and assume that $L^{2}(\nu)$ is a Riesz sequence with spectrum in $\Lambda$, giving
		
		\begin{align*}
			\norm*{\sum_{k = 1}^{N} \sum_{\lambda \in \Lambda} c_{\lambda, k} e_{\lambda + t_{k}}(x)}_{L^{2}(\mu)}^{2} & = \sum_{j = 1}^{N} \int \abs*{ \sum_{k = 1}^{N} \sum_{\lambda \in \Lambda} c_{\lambda, k} e_{\lambda + t_{k}}(x) }^{2} \ d\mu_{j}(x) \\
			& = \sum_{j = 1}^{N} \int \abs*{ \sum_{k = 1}^{N} \sum_{\lambda \in \Lambda} c_{\lambda, k} e_{\lambda + t_{k}}(x + g_{j}(x)) }^{2} \ d\nu(x)
		\end{align*}
		where again the substitution $y = \pi_{j}(x)$ is handled similarly as before going to the second line above. Expanding the argument of the exponential as before leads to
		
		\begin{align*}
			\norm*{\sum_{k = 1}^{N} \sum_{\lambda \in \Lambda} c_{\lambda, k} e_{\lambda + t_{k}}(x)}_{L^{2}(\mu)}^{2} & = \sum_{j = 1}^{N} \int \abs*{ \sum_{k = 1}^{N} \left( \sum_{\lambda \in \Lambda} c_{\lambda, k} e_{\lambda + t_{k}}(x) \right) e_{t_{k}}(g_{j}(x)) }^{2} \ d\nu(x) \\
			& = \int \sum_{j = 1}^{N} \abs*{ \sum_{k = 1}^{N} \left( \sum_{\lambda \in \Lambda} c_{\lambda, k} e_{\lambda + t_{k}}(x) \right) e_{t_{k}}(g_{j}(x)) }^{2} \ d\nu(x)\\
    & = \int \|M_x ({\bf t}) {\bf x}\|^2d\nu (x)
		\end{align*}
		where ${\bf x} = \left(\sum_{\lambda \in \Lambda} c_{\lambda, 1} e_{\lambda + t_{1}}(x),....,\sum_{\lambda \in \Lambda} c_{\lambda, N} e_{\lambda + t_{N}}(x) \right)^{\top}$. The determinant condition ensures that the lower and upper bounds are respectively bounded away from $0$ and $\infty$ giving the  bounding
		
		\begin{align*}
			\norm*{\sum_{k = 1}^{N} \sum_{\lambda \in \Lambda} c_{\lambda, k} e_{\lambda + t_{k}}(x)}_{L^{2}(\mu)}^{2} & \asymp \int \sum_{k = 1}^{N} \abs*{ \sum_{\lambda \in \Lambda} c_{\lambda, k} e_{\lambda + t_{k}}(x) }^{2} \ d\nu(x) \\
			& = \int \sum_{k = 1}^{N} \abs*{ \sum_{\lambda \in \Lambda} c_{\lambda, k} e_{\lambda}(x) }^{2} \ d\nu(x)
		\end{align*}
		which follows since going to the second line above $t_{k}$ is independent of $\Lambda$. We are finally left with
		
		\begin{align*}
			\norm*{\sum_{k = 1}^{N} \sum_{\lambda \in \Lambda} c_{\lambda, k} e_{\lambda + t_{k}}(x)}_{L^{2}(\mu)}^{2} & \asymp \sum_{k = 1}^{N} \sum_{\lambda \in \Lambda} \abs*{c_{\lambda, k}}^{2} \text{,}
		\end{align*}
		 since $\Lambda$ is the spectrum for the Riesz sequence for $\nu$. By part 2 above, we know that $\mu$ is frame spectral with the shifted spectrum found here, which makes this Riesze sequence a spanning set, hence it is in fact a Riesz basis.

  \medskip
\noindent (3 $\implies$ 2) This direction is trivial since a Riesz basis with a given spectrum is already a frame with the same spectrum by definition.

  \medskip
\noindent (2 $\implies$ 1) Starting with \eqref{eqn:DesiredFrameCondition}, fix a Borel set $E \subset D$ in the $\supp{\nu}$. Let $E_{j} = \pi_{j}^{-1}(E)$, and note that they are (almost) disjoint subsets in $\supp{\mu}$. Then for all $f \in L^{2}(\mu)$, we can define
		
		\begin{equation*}
			\tilde{f} = f \cdot {\bf 1}_{\bigcup_{j = 1}^{N} \pi_{j}^{-1}(E)} = \sum_{j = 1}^{N} f \cdot {\bf 1}_{\pi_{j}^{-1}(E)} \md
		\end{equation*}
		
		\indent It then follows that $\tilde{f}(\pi_{k}^{-1}(x)) = f(\pi_{k}^{-1}x) \cdot {\bf 1}_{E}(x)$. Let ${\bf \tilde{F}}(x) = (\tilde{f}(\pi_{1}^{-1}(x) \mc \ldots \mc \tilde{f}(\pi_{N}^{-1}(x))^{\top}$ and substituting it into \eqref{eqn:DesiredFrameCondition}, and integrating, we have
		
		\begin{equation*}
			\int \left( \norm*{M_{x} ({\bf t})} {\bf \tilde{F}}^{2} - A \norm*{{\bf \tilde{F}}}^{2} \right) \ d\nu(x) \geq 0 \md
		\end{equation*}
		By the properties of the indicator function, the equation above becomes
		
		\begin{equation} \label{eqn:DesiredDetCondition}
			\int_{E} \left( \norm*{M_{x} ({\bf t}){\bf F}(x)}^{2} - A \norm*{{\bf F}}^{2} \right) \ d\nu(x) \geq 0 \md
		\end{equation}
		 Because the above result holds for all Borel subsets of $D$, then as is well known in analysis the integrand itself is nonnegative, which gives the desire determinant condition as it is equivalent by \eqref{eqn:DetEquivCondition}. \qed \\

\section{The determinant condition.}

Our goal is to establish some necessary and sufficient condition in order for $\mathfrak{D}_{x}({\bf t})$ in \eqref{eqn:DeterminantFunction} to be uniformly bounded away from $0$, particularly when $G$ is only a closed subgroup.  When $G$ is a discrete group, the condition is trivial if the multi-tiling measure is a bounded measure since there are only finitely many such determinants, so the condition holds for almost every ${\bf t}$ \cite{K15}. There has also been an intensive study when the support of the multi-tiling measure $\mu$ is an unbounded set \cite{CC,CKM}. As now $G$ is only a closed subgroup, two elements of $G$ may be as close  to each other as we wanted even the measure is bounded. This means that the determinant condition cannot be trivially satisfied like the discrete group case.   

\begin{lemma} \label{lem:Hada}
	Let $A$ be an $N \times N$ matrix of complex entries whose entries are all unimodular. Let ${\bf v}$ and ${\bf w}$ be two column vectors in $A$. Then 
	\begin{equation*}
		\abs{\det (A)} \leq N^{\frac{N - 1}{2}} \|{\bf v} - {\bf w}\|. 
	\end{equation*}
\end{lemma}

\begin{proof}
	Without loss of generality, we can assume ${\bf v}$ and ${\bf w}$ are in the first and the second column. Viewing the determinant as a function of the columns,
	
	\begin{equation*}
		\det({\bf v}, {\bf w} \mc {\bf v}_{3} \mc \ldots \mc {\bf v}_{N}) = \det({\bf v} - {\bf w} \mc {\bf w} \mc \ldots) \md
	\end{equation*}
	
	\noindent By the Hadamard's inequality,
	\begin{equation*}
		\abs*{\det(A)} \leq \|{\bf v} - {\bf w}\| \|{\bf w}\| \|{\bf v}_{3}\| \cdots \|{\bf v}_{N}\| \md 
	\end{equation*}
	\noindent But $A$ has unimodular entries, so all column vector norms are at most $N$. Hence, the desired inequality follows. \qed \\
\end{proof}

With the above lemma in hand, we can now state and prove the desired necessary condition mentioned earlier. \\

\begin{theorem} \label{thm:NecessaryCondition}
	Let $N \in {\bZp}$ and $\nu$ be a finite Borel measure.  Suppose that there exists $\epsilon_{0} > 0$ and $t_{1} \mc, \ldots \mc t_{N} \in \R^{d}$ such that
	
	\begin{equation*}
		\abs*{\mathfrak{D}_{x}({\bf t})} \geq \epsilon_{0} > 0	
	\end{equation*}

	\noindent for $\nu$ a.e. $x$. Then there exists $m>0$ such that $\min_{j \neq k} \abs*{g_{j}(x) - g_{k}(x)} \geq m$ for $\nu$ a.e. $x$.
\end{theorem}

\begin{proof}
	Suppose that the conclusion is false. Then for all $\eta > 0$, the set of all $x$ such that
	
	\begin{equation*}
		\nu \left\{ x : \min_{j \neq k} \abs*{g_{j}(x) - g_{k}(x)} < \eta \right\} > 0 \md
	\end{equation*}
	
	\noindent Now, we take $t_{1} \mc \ldots \mc t_{N} \in \R^{d}$, for each $x$ in the above set with positive measure, the $j$th and the $k$th columns of the matrix $\left( e^{2 \pi i t_{i} \cdot g_{j}(x)} \right)$, denoted by ${\bf v}$ and ${\bf w}$, satisfy
	
	\begin{align*}
		\abs*{{\bf v} - {\bf w}}^{2} \leq & \sum_{i = 1}^{N} \abs*{e^{2 \pi i t_{i} \cdot g_{j}(x)} - e^{2 \pi i t_{i} \cdot g_{k}(x)}}^{2} \\
		\leq & \sum_{i = 1}^{N} (4 \pi^{2}) \abs*{t_{i}^{2}} \abs*{g_{j}(x) - g_{k}(x)}^{2} \\
		\leq & 4 \pi^{2} N \left( \max \abs*{t_{i}}^{2} \right) \cdot \eta^{2} \md
	\end{align*}
	
	By Lemma \ref{lem:Hada},
	
	\begin{equation*}
		\abs*{\det \left( e^{2 \pi i t_{i} \cdot g_{j}(x)} \right)} \leq N^{N} 4 \pi^{2} \left( \max \abs*{t_{i}}^{2} \right) \eta \md
	\end{equation*}
	
	\noindent Since $\eta$ is arbitrary, we see that the assumption is false, which completes the proof. \qed \par
\end{proof}

\subsection{Sufficient Condition} We now discuss the condition for which the converse of Theorem \ref{thm:NecessaryCondition} holds. If the multi-tiling measure is unbounded, there will be a lot of complications. When it was a multi-tile, it was studied by Cabrelli, Molter and Hare. As our goal of study is not on the unbounded measure, we will assume that the the multi-tiling measure has a bounded support. In particular, we have that 
\begin{equation} \label{eqn:BoundedG}
	\max_{i \neq j} \abs*{g_{j}(x) - g_{k}(x)} \leq M \mc \mbox{ for $\nu$ a.e. $x$.} \\
\end{equation}




Let $E$ be the set of all $x$ for which the above property holds. We let $G = H \oplus \Gamma$ as in Lemma \ref{lem:RdSubgroup}. Define 
\begin{equation*}
	\mathcal{M} = \{ g_{j}(x) - g_{k}(x) : j \neq k \mc x \in E \} \md
\end{equation*} \\
 By our assumption in (\ref{eqn:BoundedG}), ${\mathcal M}$ is a bounded subset of $G$. We may assume that 
\begin{equation*}
	\mathcal{M} \subset \bigcup_{i = 0}^{R-1} H_{i} \oplus \{ \gamma_{i} \} := {\mathcal K}
\end{equation*} \\
 where $H_{i}$ are compact subsets in $H$ and $\gamma_{i} \in \Gamma$ for all $i = 0 \mc 1 \mc \ldots \mc R - 1$. \\

\begin{theorem} \label{thm:SufficiencyCondition}
	Let $N \in \bZp$ and $\nu$ be a finite Borel measure and the multi-tiling measure $\mu$ is a bounded Borel measure.  Then we have the following:

	\begin{enumerate}[label=(\alph*), listparindent=1.5em]
		\item If $0 \notin \{ \gamma_{0} \mc \ldots \mc \gamma_{R - 1} \}$, then the determinant condition in Theorem \ref{thm:MainResult}(a) holds.
		
		\item If $\gamma_{0} = 0$ and there exists ${\bf w} \in H$ and $ m > 0$ such that  $\abs*{{\bf w} \cdot {\bf x}} \geq m > 0$ for all ${\bf x} \in H_{0}$, then the determinant condition in Theorem \ref{thm:MainResult}(a) holds.
			
		\item If $\dim H = 1$, then the converse of Theorem \ref{thm:NecessaryCondition} holds. \\\
	\end{enumerate}
\end{theorem}

In Example \ref{eg:3D}, we will show that despite being a bounded measure in (\ref{eqn:BoundedG}), the converse of Theorem \ref{thm:NecessaryCondition} is not true in general if $H$ has dimension more than one.  The condition in Theorem \ref{thm:SufficiencyCondition} (b) cannot be removed as well.

\begin{lemma} \label{lem:VectorExistence}
	Assume the condition in (a) or (b) in the Theorem \ref{thm:SufficiencyCondition}, there exists $\epsilon_{1} > 0$ and a non-zero vector ${\bf v}$ such that
	
	\begin{equation*}
		0 < \epsilon_{1} \leq \abs*{{\bf v} \cdot {\bf x}} \leq \frac{1}{2}
	\end{equation*}

	\noindent for all ${\bf x} \in {\mathcal K}$.
\end{lemma}

\begin{proof}
	Suppose that the condition of (a) holds. We take any ${\bf u} \in H^{\perp}$ such that ${\bf u} \cdot \gamma_{i} \neq 0$ for all $i$. This is possible since none of the $\gamma_{i}$ is zero and there are only finitely many $\gamma_{i}$. Let
	
	\begin{equation*}
		D = \max_{i = 0 \mc \ldots \mc R-1} \abs*{{\bf u} \cdot \gamma_{i}} \ \ \text{and} \ \ C = \min_{i = 0 \mc 1 \mc \ldots \mc R - 1} \abs*{{\bf u} \cdot \gamma_{i}} \md
	\end{equation*}
	
	\noindent Define ${\bf v} = \frac{1}{2D}{\bf u}$. Then for all ${\bf x} \in \mathcal{K}$, $\abs*{{\bf v} \cdot {\bf x}} = \abs*{{\bf v} \cdot \gamma_{i}}$ for some $i$, so
	
	\begin{equation*}
		0 < \frac{C}{2D} \leq \abs*{{\bf v} \cdot {\bf x}} \leq \frac{1}{2} \md
	\end{equation*}
	
	\noindent This shows that the conclusion holds. \\
	
	Suppose that the condition of (b) holds. We then take any ${\bf u} \in H^{\perp}$ such that ${\bf u} \cdot \gamma_{i} \neq 0$ for all $i = 1 \mc \ldots \mc R-1$. Let 
	
	\begin{equation*}
		D = \max_{i = 1 \mc \ldots \mc R-1} \abs*{{\bf u} \cdot \gamma_{i}} \ \ \mbox{and} \ \ C = \min_{i = 1 \mc \ldots \mc R-1} \abs*{{\bf u} \cdot \gamma_{i}} \md
	\end{equation*}
	
	\noindent Let $\epsilon > 0$ and ${\bf w}$ be defined in the assumption. Consider 
	\begin{equation*}
		F(\epsilon \mc {\bf x}) = \abs*{\left( \frac{1}{4D} {\bf u} + \epsilon {\bf w} \right) \cdot {\bf x}} \md
	\end{equation*}
Notice that $F$ is a continuous function on $\R^+\times {\mathcal K}$. We have the following claim: 

 \medskip
 
	\noindent{\it Claim:} There exists $\epsilon_{0} > 0$ such that for all $0 < \epsilon < \epsilon_{0}$, there exists $\eta>0$ such that $F(\epsilon \mc {\bf x}) \geq \eta$ for all ${\bf x} \in {\mathcal K}$. 

\medskip

	This will justify the lemma by taking ${\bf v} = \frac{1}{4D}{\bf u} + \epsilon {\bf w}$ since the claim implies the lower bound while we can choose $\epsilon$ sufficiently small so that $\epsilon \abs*{{\bf w} \cdot h_{i}} < \frac{C}{4D}$ for all $h_{i} \in H_{i}$, $i = 0 \mc 1 \mc \ldots \mc R-1$. By writing ${\bf x} = \gamma_{i} + h_{i}$, we have 
	
	\begin{equation*}
		\abs*{F(\epsilon \mc {\bf x})} = \abs*{\frac{1}{4D} {\bf u} \cdot \gamma_{i} + \epsilon{\bf w} \cdot h_{i}} \leq \frac{C}{4D} + \frac{C}{4D} = \frac{C}{2D} < \frac{1}{2} \md \\
	\end{equation*}
	
	The proof of this claim is divided into two cases where ${\mathcal K} = H_{0} \cup {\mathcal K}_{1}$ and $\mathcal{K}_{1} = \bigcup_{i = 1}^{R-1} (H_{i} + \gamma_{i})$. \\

	\noindent{\bf Case (1)}: ${\bf x} \in {\mathcal K}_1$. We write 
	
	\begin{equation*}
		F(\epsilon \mc {\bf x}) = \abs*{\frac{1}{4D} {\bf u} \cdot \gamma_{i} + \epsilon {\bf w} \cdot {\bf x}} \md
	\end{equation*}
	
	Note that $F(0 \mc {\bf x}) \geq \frac{C}{4D}$. By the continuity of $F$, we can find a neighborhood  $[0 \mc \epsilon_{x}] \times B({\bf x} \mc \delta_{{\bf x}})$ such that 
	
	\begin{equation} \label{eqn:F}
		F(\epsilon \mc {\bf x}') \geq \frac{C}{8D}, \ \ \ \forall (\epsilon \mc {\bf x}') \in [0 \mc \epsilon_{x}] \times B({\bf x} \mc \delta_{{\bf x}}) \md
	\end{equation}
	
	Observe that the collection of balls $\{ B({\bf x} \mc \delta_{{\bf x}}): {\bf x} \in H_{i} + \gamma_{i} \mc i = 1 \mc \ldots \mc R - 1 \}$ covers the compact set $\mathcal{K}_{1}$,  so we can find finitely many balls 
	\begin{equation*}
		\bigcup_{j = 1}^{N} B({\bf x}_{j} \mc \delta_{{\bf x}_{j}}) \supset \mathcal{K}_{1} \md
	\end{equation*}
	Take $\epsilon_{0} = \min \{ \epsilon_{{\bf x}_{j}}: j = 1 \mc \ldots \mc N \} > 0$, then whenever $\epsilon < \epsilon_{0}$, for all ${\bf x} \in \mathcal{K}_{1}$, ${\bf x} \in B({\bf x}_{j} \mc \delta_{{\bf x}_{j}})$ for some $j = 1 \mc \ldots \mc N$. As $\epsilon \leq \epsilon_{j}$, (\ref{eqn:F}) implies that 
	\begin{equation*}
		F(\epsilon \mc {\bf x}) \geq C/8D \md
	\end{equation*}

	\noindent{\bf Case (2)}: ${\bf x} \in H_{0}$. $F(\epsilon \mc {\bf x}) = \epsilon \abs*{{\bf w} \cdot {\bf x}} \geq \epsilon m$ by the assumption. Hence, if $\epsilon < \epsilon_{0}$ for $\epsilon_{0}$ as in Case (1), we can take $\eta = \min \{ C/8D \mc \epsilon m \}$. Hence, this completes the proof of the lemma. \qed 
\end{proof}

\medskip

\noindent{\bf Proof of Theorem \ref{thm:SufficiencyCondition}}. We take $t_{j} = (j - 1) {\bf v}$ for $j = 1 \mc \ldots \mc N$ where ${\bf v}$ is the vector in Lemma \ref{lem:VectorExistence}. Then $\left( e^{2 \pi i t_{j} \cdot g_{k}(x)} \right)$ is a Vandermonde matrix. Hence, 

\begin{equation*}
	\abs*{\det \left(e^{2 \pi i  j {\bf v} \cdot g_{k}(x)} \right)} = \prod_{j \neq k} \abs*{e^{2 \pi i {\bf v} \cdot g_{j}(x)} - e^{2 \pi i {\bf v} \cdot g_{k}(x)}} = \prod_{j \neq k} 2 \abs*{\sin(\pi {\bf v} \cdot g_{j}(x) - g_{k}(x))} \md
\end{equation*}

\noindent As $g_{j}(x) - g_{k}(x)\in {\mathcal M} \subset {\mathcal K}$, we can have the inequality in Lemma \ref{thm:SufficiencyCondition}. Using $\abs*{\sin x} \geq \frac{2}{\pi} \abs*{x}$ for all $x \in [-\frac{\pi}{2} \mc \frac{\pi}{2}]$, we have 

\begin{equation*}
	\abs*{\sin(\pi {\bf v} \cdot g_{j}(x) - g_{k}(x))} \geq 2 \abs*{{\bf v} \cdot g_{j}(x) - g_{k}(x)} \geq 2 \epsilon_{1} \md
\end{equation*}

\noindent Therefore, the determinant condition holds with $\epsilon_{0} = \prod_{j \neq k} (4 \epsilon_{1}) = (4 \epsilon_{1})^{N(N-1)/2}$, which shows (a) and (b). \\

We now prove part (c). Suppose that $\dim H = 1$. We can write $H = \{ t {\bf v}: t \in \R \}$ and ${\bf v}$ is a unit vector. Then the assumption on the set ${\mathcal M}$ implies that $H_{0} \subset \{ t {\bf v}: t \in  [-M \mc M] \setminus [-m \mc m] \}$ for some $M \mc m > 0$. Then any non-zero ${\bf w} \in H$ will satisfy the assumption in (ii), so we have completed the proof. \qed \par

\section{Revisiting the Square Boundary and other Examples}

In this section, we will apply our theory to several examples and illustrate the sharpness of the conditions given in Theorem \ref{thm:SufficiencyCondition}. 
\medskip

\noindent{\it Proof of Theorem \ref{thm:BndrySqResult}.}  The embedding space here is $\R^{2}$ with $G$ being the closed additive group $\Z \times \R$ whose fundamental domain $D$ is $[0 \mc 1] \times \{ 0 \}$ and $\nu$ being the singular Lebesgue measure $d \mu(x \mc y) = \chi_{[0 \mc 1]} \ dx \times \delta_0$. This measure admits a  spectrum $\Lambda = \Z \times \{ 0 \}$. 	The boundary of the connected square is then defined by the Lebesgue measure $\mu = \mu_{1} + \mu_{2}$, where $\mu_{1}$ is supported on $\{ (x \mc x) : x \in [0 \mc 1 / 2] \} \cup \{ (x \mc 1 - x) : x \in (1/2 \mc 1] \}$ and $\mu_{2}$ is supported on $\{ (x \mc - x) : x \in [0 \mc 1/2] \} \cup \{ (x \mc x - 1) : x \in (1/2 \mc 1] \}$. Both $\mu_{1}$ and $\mu_{2}$ are then easily identified as equivalent to $\nu$ via the respective translations collected in $\mathcal{G}_{(x \mc 0)}$ from \eqref{eqn:PbMapElementCollection} with $\mathcal{G}_{(x \mc 0)} = \{ g_{1}(x \mc 0) \mc g_{2}(x \mc 0) \}$, where
	
	\begin{align*}
		g_{1}(x \mc 0) & =
		\begin{cases}
			(0 \mc x) & \text{if } x \in [0 \mc \frac{1}{2}] \mc \\
			(0 \mc 1 - x) & \text{if } x \in (\frac{1}{2} \mc 1]
		\end{cases} \mc \\
		g_{2}(x \mc 0) & = 
		\begin{cases}
			(0 \mc -x) & \text{if } x \in [0 \mc \frac{1}{2}] \mc \\
			(0 \mc x - 1) & \text{if } x \in (\frac{1}{2} \mc 1]
		\end{cases} \md
	\end{align*}

	Observe then that
	
	\begin{equation*}
		 \abs*{g_{1}(x \mc 0) - g_{2}(x \mc 0)} = 2 x \mc
	\end{equation*}
	
	\noindent which goes to $0$ as $x$ goes to $0$. Hence, by the contrapositive of Theorem \ref{thm:NecessaryCondition}, there is no collection $t_{1} \mc t_{2} \in R^{2}$ for which $\abs*{\mathfrak{D}_{x}(t)}$ is uniformly bounded away from $0$ for almost all $x$. Thus, by Theorem \ref{thm:MainResult}, no exponential Riesz basis whose spectrum is given by $\cup_{k = 1}^{2} (\Lambda + t_{k})$ exists.  \qed \par

	\medskip

\begin{Example} \label{eg:SqBdry}
	{\rm On the other hand, we can define a measure supported on a separated square boundary where a vertical separation occurs at $(0 \mc 0)$ and $(1 \mc 0)$ and the top and bottom triangular vertices are translated accordingly (see Figure \ref{fig:SqBdryRB}).}

	\begin{figure}[ht!]
		\begin{center}
			\resizebox{8cm}{!}
			{
				\begin{tikzpicture}[decoration={markings,
					}]
					\draw[line width=1.5pt,-] (-5,0) -- (5,0) coordinate (xaxis);
					\draw[line width=1.5pt,-] (0,-3) -- (0,3) coordinate (yaxis);
					
					\foreach \x in {-3,0,3}
					{
						\draw[dashed] (\x , -3) -- (\x, 3);
					}
					
					\coordinate (0) at (0, 0);
					
					\path[draw,line width=1pt,postaction=decorate] (0, 0.3) -- (1.5, 1.8) -- (3, 0.3);
					\path[draw,line width=1pt,postaction=decorate] (3, -0.3) -- (1.5, -1.8) -- (0, -0.3);
					
					\node[below] at (xaxis) {$x$};
					\node[left] at (yaxis) {$y$};
					\node[below left] {$O$};
					\node[below right] at (3, 0) {$1$};
					\node[below left] at (-3, 0){$-1$};
					\node[above] at (1.5, 1.8){$\mu_{1}$};
					\node[below] at (1.5, -1.8) {$\mu_{2}$};
					\node[left] at (3.2, 0) {$\left\{ \right.$};
					\node[above left] at (3, 0) {$\delta$};
				\end{tikzpicture}
			}
			\caption{The boundary of the square whose top and bottom triangular-half pieces are separated by a minimum vertical distance $\delta$ overlaid with $G = \Z \times \R$. On this (separated) boundary, a exponential Riesz basis is admitted.}
			\label{fig:SqBdryRB}
		\end{center}
	\end{figure}
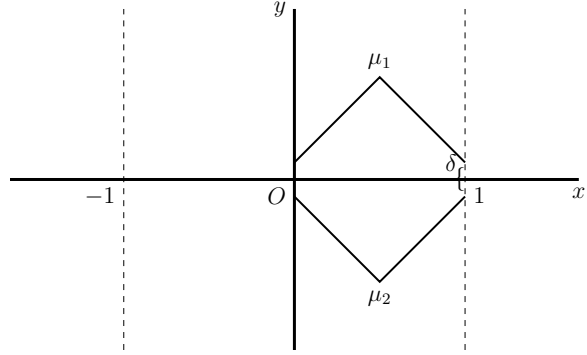

{\rm For the separated boundary, we have the Lebesgue measure $\mu = \mu_{1} + \mu_{2}$, where $\mu_{1}$ is supported on $\{ (x \mc x + \delta/2) : x \in [0 \mc 1 / 2] \} \cup \{ (x \mc 1 - x + \delta/2) : x \in (1/2 \mc 1] \}$ and $\mu_{2}$ is supported on $\{ (x \mc - x - \delta/2) : x \in [0 \mc 1/2] \} \cup \{ (x \mc x - 1 - \delta/2) : x \in (1/2 \mc 1] \}$. Here, we have $\mathcal{G}_{x \mc 0} = \{ g_{1}(x \mc 0) \mc g_{2}(x \mc 0) \}$, where}
	
	\begin{align*}
		g_{1}(x \mc 0) & =
		\begin{cases}
			(0 \mc x + \frac{\delta}{2}) & \text{if } x \in [0 \mc \frac{1}{2}] \mc \\
			(0 \mc 1 - x + \frac{\delta}{2}) & \text{if } x \in (\frac{1}{2} \mc 1]
		\end{cases} \mc \\
		g_{2}(x \mc y) & = 
		\begin{cases}
			(0 \mc -x - \frac{\delta}{2}) & \text{if } [0 \mc \frac{1}{2}] \mc \\
			(0 \mc x - 1 - \frac{\delta}{2}) & \text{if } x \in (\frac{1}{2} \mc 1]
		\end{cases} \md
	\end{align*}
{\rm We have}
	\begin{equation*}
		\abs*{g_{1}(x \mc 0) - g_{2}(x \mc 0)} = 2 x + \delta \geq \delta > 0 \md
	\end{equation*}
	
	\noindent{\rm  By Theorem \ref{thm:SufficiencyCondition} (c), we have that the converse of Theorem \ref{thm:NecessaryCondition} holds. Hence, the determinant condition holds with some $t_{1} \mc t_{2} \in {\mathbb R}^{2}$, so by Theorem \ref{thm:MainResult}, we have an exponential Riesz basis whose spectrum is of the form $\cup_{k = 1}^{2} (\Lambda + t_{k})$.}
\end{Example}

In all previous examples, the determinant condition failed because two fundamental domains got too close to each other and makes the condition $\min_{j\ne k}|g_j(x)-g_k(x)|\ge m >0$ a.e. fail in Theorem \ref{thm:NecessaryCondition}. In the following example, two fundamental domains are away from each other, yet the determinant condition fails and hence there is no structured Riesz basis of exponentials. 

\begin{Example} \label{eg:3D}
{\rm On $\R^3$, we consider the group $G = \Z\times \R^2$. Then a fundamental domain of $G$ is $D = [0,1]\times \{(0,0)\}$ and let $\nu$ be the arc length measure on $D$. Let }
$$
X = D \cup \{(x,\cos(2\pi x),\sin (2\pi x))): x\in[0,1]\}.
$$
{\rm Define also the arc-length measure on $\{(x,\cos(2\pi x),\sin (2\pi x))): x\in[0,1]\}$ by $\nu'$.
Then $\mu = \nu + \nu' $ is a multi-tiling measure and $X$ is a union of two fundamental domains of $G$ (See the Figure below). }
   \begin{figure}[h]
\begin{center}
	\includegraphics[width=10cm]{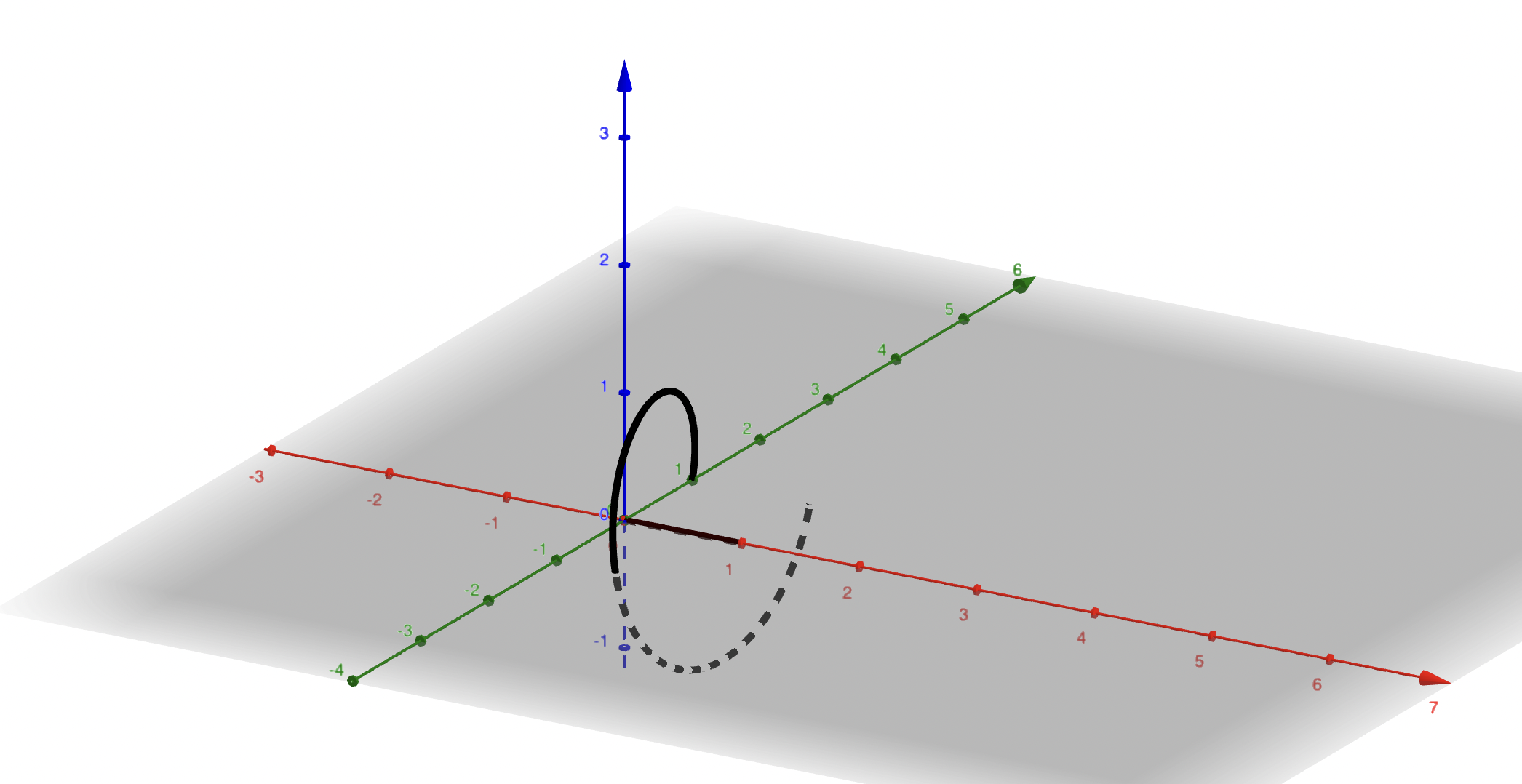}
\end{center}
\end{figure}
{\rm By a direct computation, for each $(x,0,0)\in D$, }
$$
{\mathcal G}_{(x,0,0)} = \{ (0,0,0), (0, \cos (2\pi x),\sin (2\pi x))\}
$$
{\rm Therefore,  $|g_1(x,0,0)-g_2(x,0,0)| = 1$ and hence  $\min_{j\ne k}|g_j(x,0,0)-g_k(x,0,0)| = 1>0$. However, for all ${\bf t} = ({\bf t}_1,{\bf t}_2)\in \R^{6}$, }
$$
{\mathfrak D}_{(x,0,0)}({\bf t}) = \det \begin{bmatrix}
    1 & e^{2\pi i {\bf t}_1\cdot (0,\cos(2\pi x),\sin(2\pi x))} \\ 1 & e^{2\pi i {\bf t}_2\cdot (0,\cos(2\pi x),\sin(2\pi x))}\\
\end{bmatrix} = e^{2\pi i (({\bf t}_1-{\bf t}_2)) \cdot (0,\cos (2\pi x),\sin (2\pi x))}-1.
$$
{\rm Note that each projection of ${\bf t}_1-{\bf t}_2$ onto the $yz$-plane, we can always find a unit vector $(\cos(2\pi x), \sin(2\pi x))$ such that $({\bf t}_1-{\bf t}_2)\cdot (0,\cos(2\pi x), \sin(2\pi x)) = 0$. Therefore, for this $x$, ${\mathfrak D}_{(x,0,0)}({\bf t}) = 0$. Hence, by Theorem \ref{thm:MainResult}, there is no structured Riesz basis of exponentials $\bigcup_{i=1,2} (\Lambda+{\bf t}_i)$ where $\Lambda = \Z\times \{(0,0)\}$. }

{\rm In this example, the condition in Theorem \ref{thm:SufficiencyCondition}(b) also does not hold. }

\end{Example}

\end{document}